\documentclass[12pt, reqno]{amsart}
\usepackage{amsmath, amsthm, amscd, amsfonts, amssymb, graphicx, color}
\usepackage[bookmarksnumbered, colorlinks, plainpages]{hyperref}

\newtheorem{theorem}{Theorem}[section]
\newtheorem{lemma}[theorem]{Lemma}

\theoremstyle{definition}

\theoremstyle{remark}

\numberwithin{equation}{section}
\input{tcilatex}

\begin{document}
\title[Wright Funtions]{Partial Sums of Normalized Wright Functions}
\author[M. U. Din, M. Raza, N. Ya\u{g}mur ]{Muhey u Din$^{1},$ Mohsan Raza$%
^{2}$, Nihat Ya\u{g}mur$^{3\ast }$ }
\address{$^{1}$Department of Mathematics, Government College University
Faisalabad, Pakistan}
\email{muheyudin@gmail.com.}
\address{$^{2}$ Department of Mathematics, Government College University
Faisalabad, Pakistan.}
\email{mohsan976@yahoo.com}
\address{$^{3}$Department of Mathematics, Erzincan University, Erzincan
24000, Turkey}
\email{nhtyagmur@gmail.com}
\date{\\
\indent$^{\ast }$ Corresponding author\\
2010\textit{\ Mathematics Subject Classification. }30C45, 30C50, 33C10.}

\begin{abstract}
In this paper we find the partial sums of two kinds normalized Wright
functions and the partial sums of Alexander transform of these normalized
Wright functions.
\end{abstract}

\maketitle

\setcounter{page}{1}

%--------------------------------------------------------------------------------------------------------------

%\dedicatory{This paper is dedicated to Professor ABCD}

\section{Introduction and preliminaries}

Let $\mathcal{A}$ be the class of functions $f$ of the form%
\begin{equation*}
f(z)=z+\dsum\limits_{m=2}^{\infty }a_{m}z^{m}
\end{equation*}%
analytic in the open unit disc $\mathcal{U}=\left\{ z:\left\vert
z\right\vert <1\right\} $. Consider the Alexander transform given as:%
\begin{equation*}
\mathbb{A}\left[ f\right] \left( z\right) =\dint\limits_{0}^{z}\frac{f(t)}{t}%
dt=z+\dsum\limits_{m=2}^{\infty }\frac{a_{m}}{m}z^{m}.
\end{equation*}%
The surprize use of Hypergeometric function in the solution of the
Bieberbach conjecture has attracted many researchers to study the special
functions. Many authors who study on geometric functions theory are
intersted in some geometric properties such as univalency, starlikeness,
convexity and close-to-convexity of special functions. Recently, several
researchers have studied the geometric properties of hypergeometric
functions \cite{mil, rus}, Bessel functions \cite{BJ1, BJ2, barciz b, BP,
BM, b1, s, sk, k}, Struve functions \cite{OY,YO}, Lommel functions \cite{MC}%
. Motivated by the above works Prajpat \cite{Pr} studied some geometric
properties of Wright function%
\begin{equation*}
W_{\lambda ,\mu }(z)=\dsum\limits_{m=0}^{\infty }\frac{z^{m}}{m!\Gamma
\left( \lambda m+\mu \right) },\text{ }\lambda >-1,\text{ }\mu \in 
%TCIMACRO{\U{2102} }%
%BeginExpansion
\mathbb{C}
%EndExpansion
.
\end{equation*}%
This series is absolutely convergent in $%
%TCIMACRO{\U{2102} }%
%BeginExpansion
\mathbb{C}
%EndExpansion
,$ when $\lambda >-1$ and absolutely convergent in open unit disc $\mathcal{U%
}$ for $\lambda =-1.$ Furthermore this function is entire. The Wright
functions were introduced by Wright \cite{wright} and have been used in the
asymtotic theory of partitions, in the theory of integral transforms of the
Hankel type and in Mikusinski operational calculus. Recently, Wright
functions have been found in the solution of partial differential equations
of fractional order. It was found that the corresponding Green functions can
be represented in terms of the Wright functions \cite{pod,samko}$.$ For
positive rational number $\lambda ,$ the Wright functions can be represented
in terms of generalized hypergeometric functions. For some details see \cite[%
section 2.1]{gor}. In particular, the functions $W_{1,v+1}(-z^{2}/4)$ can be
expressed in terms of the Bessel functions $J_{v}$, given as:%
\begin{equation*}
J_{v}\left( z\right) =\left( \frac{z}{2}\right)
^{2}W_{1,v+1}(-z^{2}/4)=\dsum\limits_{m=0}^{\infty }\frac{\left( -1\right)
^{m}\left( z/2\right) ^{2m+v}}{m!\Gamma \left( m+v+1\right) }.
\end{equation*}%
The Wright function generalizes various functions like Array functions,
Whittaker functions, entire auxiliary functions, etc. For the details, we
refer to \cite{gor}$.$ Prajapat discussed some geometric properties of the
following normalizations of Wright functions in \cite{Pr}\ 
\begin{eqnarray}
\mathcal{W}_{\lambda ,\mu }(z) &=&\Gamma \left( \mu \right) zW_{\lambda ,\mu
}(z)  \notag \\
&=&z+\dsum\limits_{m=1}^{\infty }\frac{\Gamma \left( \mu \right) }{m!\Gamma
\left( \lambda m+\mu \right) }z^{m+1}.\text{ }\lambda >-1,\text{ }\mu >0,%
\text{ }z\in \mathcal{U}\text{,}  \label{m3} \\
\mathbb{W}_{\lambda ,\mu }(z) &=&\Gamma \left( \lambda +\mu \right) \left[
W_{\lambda ,\mu }(z)-\frac{1}{\Gamma \left( \mu \right) }\right]  \notag \\
&=&z+\dsum\limits_{m=1}^{\infty }\frac{\Gamma \left( \lambda +\mu \right) }{%
\left( m+1\right) !\Gamma \left( \lambda m+\lambda +\mu \right) }z^{m+1},%
\text{ \ }z\in \mathcal{U}\text{,}  \label{md}
\end{eqnarray}%
where $\lambda >-1,\lambda +\mu >0.$ The Pochhammer (or Appell) symbol,
defined in terms of Euler's gamma functions is given as $(x)_{n}=\Gamma
(x+n)/\Gamma (x)=x(x+1)...(x+n-1)$.

In this note, we study the ratio of a function of the forms $\left( \ref{m3}%
\right) $ and$\ \left( \ref{md}\right) $ to its sequence of partial sums $%
\left( \mathcal{W}_{\lambda ,\mu }\right) _{n}(z)=z+\dsum\limits_{m=1}^{n}%
\frac{\Gamma \left( \mu \right) }{m!\Gamma \left( \lambda m+\mu \right) }%
z^{m+1}$ when the coefficients of $\mathcal{W}_{\lambda ,\mu }$ satisfy
certain conditions. We determine the lower bounds of \ $\func{Re}\left\{ 
\frac{\mathcal{W}_{\lambda ,\mu }(z)}{\left( \mathcal{W}_{\lambda ,\mu
}\right) _{n}(z)}\right\} ,$ $\func{Re}\left\{ \frac{\left( \mathcal{W}%
_{\lambda ,\mu }\right) _{n}(z)}{\mathcal{W}_{\lambda ,\mu }(z)}\right\} ,$ $%
\func{Re}\left\{ \frac{\mathcal{W}_{\lambda ,\mu }^{\prime }(z)}{\left( 
\mathcal{W}_{\lambda ,\mu }\right) _{n}^{\prime }(z)}\right\} ,$ $\func{Re}%
\left\{ \frac{\left( \mathcal{W}_{\lambda ,\mu }\right) _{n}^{\prime }(z)}{%
\mathcal{W}_{\lambda ,\mu }^{\prime }(z)}\right\} ,$

$\func{Re}\left\{ \frac{\mathbb{A}\left[ \mathcal{W}_{\lambda ,\mu }\right]
(z)}{\left( \mathbb{A}\left[ \mathcal{W}_{\lambda ,\mu }\right] \right)
_{n}(z)}\right\} ,$ $\func{Re}\left\{ \frac{\left( \mathbb{A}\left[ \mathcal{%
W}_{\lambda ,\mu }\right] \right) _{n}(z)}{\mathbb{A}\left[ \mathcal{W}%
_{\lambda ,\mu }\right] (z)}\right\} $, where $\mathbb{A}\left[ \mathcal{W}%
_{\lambda ,\mu }\right] $ is the Alexander transform of $\mathcal{W}%
_{\lambda ,\mu }.$ Some similar results are obtained for the function $%
\mathbb{W}_{\lambda ,\mu }(z).$ For some works on partial sums, we refer 
\cite{Brik, lin, orhG, Jmi, owa, small, silver, silva}.

\begin{lemma}
\label{1} Let $\lambda ,$ $\mu \in 
%TCIMACRO{\U{211d} }%
%BeginExpansion
\mathbb{R}
%EndExpansion
$ and $\lambda >-1,$ $\mu >0.$ Then the function $\mathcal{W}_{\lambda ,\mu
}:\mathcal{U\rightarrow 
%TCIMACRO{\U{2102} }%
%BeginExpansion
\mathbb{C}
%EndExpansion
\ }$ defined by $\left( \ref{m3}\right) $ satisfies the following
inequalities:

(i) If $\mu >\frac{1}{2},$ then 
\begin{equation*}
\left\vert \mathcal{W}_{\lambda ,\mu }(z)\right\vert \leq \frac{2\mu +1}{%
2\mu -1},\ \ z\in \mathcal{U}.
\end{equation*}%
(ii) If $\mu >1,$ then 
\begin{equation*}
\left\vert \mathcal{W}_{\lambda ,\mu }^{\prime }(z)\right\vert \leq \frac{%
\mu +1}{\mu -1},\ \ z\in \mathcal{U}.
\end{equation*}%
(iii) If $\mu >\frac{1}{2},$ then%
\begin{equation*}
\left\vert \mathbb{A}\left[ \mathcal{W}_{\lambda ,\mu }\right]
(z)\right\vert \leq \frac{2\mu }{2\mu -1},\ \ z\in \mathcal{U}.
\end{equation*}
\end{lemma}

\begin{proof}
(i) By using the well-known triangle inequalitiy%
\begin{equation*}
\left\vert z_{1}+z_{2}\right\vert \leq \left\vert z_{1}\right\vert
+\left\vert z_{2}\right\vert
\end{equation*}%
with the inequality $\Gamma \left( \mu +m\right) \leq \Gamma \left( \mu
+m\lambda \right) ,$ $m\in 
%TCIMACRO{\U{2115} }%
%BeginExpansion
\mathbb{N}
%EndExpansion
,$ which is equivalent to $\frac{\Gamma \left( \mu \right) }{\Gamma \left(
\lambda m+\mu \right) }\leq \frac{1}{\mu (\mu +1)...(\mu +m-1)}=\frac{1}{%
\left( \mu \right) _{m}},$ $m\in 
%TCIMACRO{\U{2115} }%
%BeginExpansion
\mathbb{N}
%EndExpansion
$ and the inequalities%
\begin{equation*}
\left( \mu \right) _{m}\geq \mu ^{m},\ m!\geq 2^{m-1},\ m\in 
%TCIMACRO{\U{2115} }%
%BeginExpansion
\mathbb{N}
%EndExpansion
,
\end{equation*}%
we obtain%
\begin{eqnarray*}
\left\vert \mathcal{W}_{\lambda ,\mu }(z)\right\vert &=&\left\vert
z+\dsum\limits_{m=1}^{\infty }\frac{\Gamma \left( \mu \right) }{m!\Gamma
\left( \lambda m+\mu \right) }z^{m+1}\right\vert \leq
1+\dsum\limits_{m=1}^{\infty }\frac{\Gamma \left( \mu \right) }{m!\Gamma
\left( \lambda m+\mu \right) } \\
&\leq &1+\dsum\limits_{m=1}^{\infty }\frac{1}{m!\left( \mu \right) _{m}} \\
&\leq &1+\frac{1}{\mu }\dsum\limits_{m=1}^{\infty }\left( \frac{1}{2\mu }%
\right) ^{m-1} \\
&=&\frac{2\mu +1}{2\mu -1},\text{ \ }\mu >1/2,\ \ \ z\in \mathcal{U}\text{.}
\end{eqnarray*}%
(ii) To prove \ (ii), we use the well-known triangle inequality with the
inequality $\frac{\Gamma \left( \mu \right) }{\Gamma \left( \lambda m+\mu
\right) }\leq \frac{1}{\mu (\mu +1)...(\mu +m-1)}=\frac{1}{\left( \mu
\right) _{m}},$ $m\in 
%TCIMACRO{\U{2115} }%
%BeginExpansion
\mathbb{N}
%EndExpansion
$ and the inequalities%
\begin{equation*}
\left( \mu \right) _{m}\geq \mu ^{m},\ m!\geq \frac{m+1}{2},\ m\in 
%TCIMACRO{\U{2115} }%
%BeginExpansion
\mathbb{N}
%EndExpansion
,
\end{equation*}%
we have 
\begin{eqnarray*}
\left\vert \mathcal{W}_{\lambda ,\mu }^{\prime }(z)\right\vert &=&\left\vert
1+\dsum\limits_{m=1}^{\infty }\frac{\Gamma \left( \mu \right) (m+1)}{%
m!\Gamma \left( \lambda m+\mu \right) }z^{m}\right\vert \leq
1+\dsum\limits_{m=1}^{\infty }\frac{\Gamma \left( \mu \right) (m+1)}{%
m!\Gamma \left( \lambda m+\mu \right) } \\
&\leq &1+\dsum\limits_{m=1}^{\infty }\frac{m+1}{m!\left( \mu \right) _{m}} \\
&\leq &1+\frac{2}{\mu }\dsum\limits_{m=1}^{\infty }\left( \frac{1}{\mu }%
\right) ^{m-1} \\
&=&\frac{\mu +1}{\mu -1},\text{ \ \ }\mu >1,\ z\in \mathcal{U}\text{.}
\end{eqnarray*}%
(iii) Making the use of triangle inequality with $\frac{\Gamma \left( \mu
\right) }{\Gamma \left( \lambda m+\mu \right) }\leq \frac{1}{\left( \mu
\right) _{m}}$ and the inequalities%
\begin{equation*}
\left( \mu \right) _{m}\geq \mu ^{m},\ \left( m+1\right) !\geq 2^{m},\ m\in 
%TCIMACRO{\U{2115} }%
%BeginExpansion
\mathbb{N}
%EndExpansion
,
\end{equation*}%
we have%
\begin{eqnarray*}
\left\vert \mathbb{A}\left[ \mathcal{W}_{\lambda ,\mu }\right]
(z)\right\vert &=&\left\vert z+\dsum\limits_{m=1}^{\infty }\frac{\Gamma
\left( \mu \right) }{\left( m+1\right) !\Gamma \left( \lambda m+\mu \right) }%
z^{m+1}\right\vert \\
&\leq &1+\dsum\limits_{m=1}^{\infty }\frac{\Gamma \left( \mu \right) }{%
\left( m+1\right) !\Gamma \left( \lambda m+\mu \right) } \\
&\leq &1+\dsum\limits_{m=1}^{\infty }\frac{1}{\left( m+1\right) !\left( \mu
\right) _{m}} \\
&\leq &1+\frac{1}{2\mu }\dsum\limits_{m=1}^{\infty }\left( \frac{1}{2\mu }%
\right) ^{m-1} \\
&=&\frac{2\mu }{2\mu -1},\text{ \ \ }\mu >1/2,\ \ \ \ z\in \mathcal{U}\text{.%
}
\end{eqnarray*}
\end{proof}

\begin{lemma}
\label{2}Let $\lambda ,$ $\mu \in 
%TCIMACRO{\U{211d} }%
%BeginExpansion
\mathbb{R}
%EndExpansion
$ and $\lambda >-1,$ $\lambda +\mu >0.$ Then the function $\mathbb{W}%
_{\lambda ,\mu }:\mathcal{U\rightarrow 
%TCIMACRO{\U{2102} }%
%BeginExpansion
\mathbb{C}
%EndExpansion
\ }$ defined by $\left( \ref{md}\right) $ satisfies the following
inequalities:

(i) If $\lambda +\mu >\frac{1}{2},$ then 
\begin{equation*}
\left\vert \mathbb{W}_{\lambda ,\mu }(z)\right\vert \leq \frac{2\left(
\lambda +\mu \right) }{2\left( \lambda +\mu \right) -1},\text{ \ \ }z\in 
\mathcal{U}.
\end{equation*}%
(ii) If $\lambda +\mu >\frac{1}{2},$ then 
\begin{equation*}
\left\vert \mathbb{W}_{\lambda ,\mu }^{\prime }(z)\right\vert \leq \frac{%
2\left( \lambda +\mu \right) +1}{2\left( \lambda +\mu \right) -1},\text{ \ \ 
}z\in \mathcal{U}.
\end{equation*}
\end{lemma}

\begin{proof}
(i) By using the well-known triangle inequality%
\begin{equation*}
\left\vert z_{1}+z_{2}\right\vert \leq \left\vert z_{1}\right\vert
+\left\vert z_{2}\right\vert
\end{equation*}%
with the inequality $\Gamma \left( \lambda +\mu +m\right) \leq \Gamma \left(
m\lambda +\lambda +\mu \right) ,$ $m\in 
%TCIMACRO{\U{2115} }%
%BeginExpansion
\mathbb{N}
%EndExpansion
,$ which is equivalent to $\frac{\Gamma \left( \lambda +\mu \right) }{\Gamma
\left( m\lambda +\lambda +\mu \right) }\leq \frac{1}{\left( \lambda +\mu
\right) (\lambda +\mu +1)...(\lambda +\mu +m-1)}=\frac{1}{\left( \lambda
+\mu \right) _{m}},$ $m\in 
%TCIMACRO{\U{2115} }%
%BeginExpansion
\mathbb{N}
%EndExpansion
$ and the inequalities%
\begin{equation*}
\left( \lambda +\mu \right) _{m}\geq \left( \lambda +\mu \right) ^{m},\
m!\geq 2^{m-1},\ m\in 
%TCIMACRO{\U{2115} }%
%BeginExpansion
\mathbb{N}
%EndExpansion
,
\end{equation*}%
we obtain%
\begin{eqnarray*}
\left\vert \mathbb{W}_{\lambda ,\mu }(z)\right\vert &=&\left\vert
z+\dsum\limits_{m=1}^{\infty }\frac{\Gamma \left( \lambda +\mu \right) }{%
m!\Gamma \left( \lambda m+\lambda +\mu \right) }z^{m+1}\right\vert \leq
1+\dsum\limits_{m=1}^{\infty }\frac{\Gamma \left( \lambda +\mu \right) }{%
m!\Gamma \left( \lambda m+\lambda +\mu \right) } \\
&\leq &1+\dsum\limits_{m=1}^{\infty }\frac{1}{m!\left( \lambda +\mu \right)
_{m}} \\
&\leq &1+\frac{1}{\lambda +\mu }\dsum\limits_{m=1}^{\infty }\left( \frac{1}{%
2\left( \lambda +\mu \right) }\right) ^{m-1} \\
&=&\frac{2\left( \lambda +\mu \right) +1}{2\left( \lambda +\mu \right) -1},%
\text{ \ }2\left( \lambda +\mu \right) >1/2,\ \ \ z\in \mathcal{U}\text{.}
\end{eqnarray*}%
(ii) By using the well-known triangle inequality with the inequality $\frac{%
\Gamma \left( \lambda +\mu \right) }{\Gamma \left( m\lambda +\lambda +\mu
\right) }\leq \frac{1}{\left( \lambda +\mu \right) (\lambda +\mu
+1)...(\lambda +\mu +m-1)}=\frac{1}{\left( \lambda +\mu \right) _{m}},m\in 
%TCIMACRO{\U{2115} }%
%BeginExpansion
\mathbb{N}
%EndExpansion
$ and the inequalities%
\begin{equation*}
\left( \lambda +\mu \right) _{m}\geq \left( \lambda +\mu \right) ^{m},\
m!\geq \frac{m+1}{2},\ m\in 
%TCIMACRO{\U{2115} }%
%BeginExpansion
\mathbb{N}
%EndExpansion
,
\end{equation*}%
we have 
\begin{eqnarray*}
\left\vert \mathbb{W}_{\lambda ,\mu }^{\prime }(z)\right\vert &=&\left\vert
1+\dsum\limits_{m=1}^{\infty }\frac{\Gamma \left( \lambda +\mu \right) (m+1)%
}{m!\Gamma \left( \lambda m+\lambda +\mu \right) }z^{m}\right\vert \leq
1+\dsum\limits_{m=1}^{\infty }\frac{\Gamma \left( \lambda +\mu \right) (m+1)%
}{m!\Gamma \left( \lambda m+\lambda +\mu \right) } \\
&\leq &1+\dsum\limits_{m=1}^{\infty }\frac{m+1}{m!\left( \lambda +\mu
\right) _{m}} \\
&\leq &1+\frac{2}{\left( \lambda +\mu \right) }\dsum\limits_{m=1}^{\infty
}\left( \frac{1}{\lambda +\mu }\right) ^{m-1} \\
&=&\frac{\left( \lambda +\mu \right) +1}{\left( \lambda +\mu \right) -1},%
\text{ \ \ }\left( \lambda +\mu \right) >1,\ z\in \mathcal{U}\text{.}
\end{eqnarray*}
\end{proof}

\section{Partial Sums of $\mathcal{W}_{\protect\lambda ,\protect\mu }(z)$}

\begin{theorem}
\label{3}Let $\lambda ,$ $\mu \in 
%TCIMACRO{\U{211d} }%
%BeginExpansion
\mathbb{R}
%EndExpansion
$ such that $\lambda >-1,$ $\mu >\frac{3}{2}.$ Then%
\begin{equation}
\func{Re}\left\{ \frac{\mathcal{W}_{\lambda ,\mu }(z)}{\left( \mathcal{W}%
_{\lambda ,\mu }\right) _{n}(z)}\right\} \geq \frac{2\mu -3}{2\mu -1}\text{,
\ \ }z\in \mathcal{U}.  \label{m4}
\end{equation}%
and%
\begin{equation}
\func{Re}\left\{ \frac{\left( \mathcal{W}_{\lambda ,\mu }\right) _{n}(z)}{%
\mathcal{W}_{\lambda ,\mu }(z)}\right\} \geq \frac{2\mu -1}{2\mu +1},\text{
\ \ }z\in \mathcal{U}.  \label{m5}
\end{equation}
\end{theorem}

\begin{proof}
By using (i) of Lemma $\ref{1}$, it is clear that 
\begin{equation*}
1+\dsum\limits_{m=1}^{\infty }\left\vert a_{m}\right\vert \leq \frac{2\mu +1%
}{2\mu -1},
\end{equation*}%
which is equivalent to%
\begin{equation*}
\frac{2\mu -1}{2}\dsum\limits_{m=1}^{\infty }\left\vert a_{m}\right\vert
\leq 1.
\end{equation*}%
where $a_{m}=\frac{\Gamma \left( \mu \right) }{m!\Gamma \left( \lambda m+\mu
\right) }.$\ Now, we may write%
\begin{eqnarray*}
&&\frac{2\mu -1}{2}\left\{ \frac{\mathcal{W}_{\lambda ,\mu }(z)}{\left( 
\mathcal{W}_{\lambda ,\mu }\right) _{n}(z)}-\frac{2\mu -3}{2\mu -1}\right\}
\\
&=&\frac{1+\dsum\limits_{m=1}^{n}a_{m}z^{m}+\left( \frac{2\mu -1}{2}\right)
\dsum\limits_{m=n+1}^{\infty }a_{m}z^{m}}{1+\dsum\limits_{m=1}^{n}a_{m}z^{m}}
\\
&=&:\frac{1+w(z)}{1-w(z)}.
\end{eqnarray*}%
Then it is clear that%
\begin{equation*}
w(z)=\frac{\left( \frac{2\mu -1}{2}\right) \dsum\limits_{m=n+1}^{\infty
}a_{m}z^{m}}{2+2\dsum\limits_{m=1}^{n}a_{m}z^{m}+\left( \frac{2\mu -1}{2}%
\right) \dsum\limits_{m=n+1}^{\infty }a_{m}z^{m}}
\end{equation*}%
and%
\begin{equation*}
\left\vert w(z)\right\vert \leq \frac{\left( \frac{2\mu -1}{2}\right)
\dsum\limits_{m=n+1}^{\infty }\left\vert a_{m}\right\vert }{%
2-2\dsum\limits_{m=1}^{n}\left\vert a_{m}\right\vert -\left( \frac{2\mu -1}{2%
}\right) \dsum\limits_{m=n+1}^{\infty }\left\vert a_{m}\right\vert }.
\end{equation*}%
This implies that $\left\vert w\left( z\right) \right\vert \leq 1$ if and
only if%
\begin{equation*}
2\left( \frac{2\mu -1}{2}\right) \dsum\limits_{m=n+1}^{\infty }\left\vert
a_{m}\right\vert \leq 2-2\dsum\limits_{m=1}^{n}\left\vert a_{m}\right\vert .
\end{equation*}%
Which further implies that 
\begin{equation}
\dsum\limits_{m=1}^{n}\left\vert a_{m}\right\vert +\left( \frac{2\mu -1}{2}%
\right) \dsum\limits_{m=n+1}^{\infty }\left\vert a_{m}\right\vert \leq 1.
\label{m6}
\end{equation}%
It suffices to show that the left hand side of $\left( \ref{m6}\right) $ is
bounded above by $\left( \frac{2\mu -1}{2}\right) \dsum\limits_{m=1}^{\infty
}\left\vert a_{m}\right\vert ,$ which is equivalent to%
\begin{equation*}
\frac{2\mu -3}{2}\dsum\limits_{m=1}^{n}\left\vert a_{m}\right\vert \geq 0.
\end{equation*}

To prove $\left( \ref{m5}\right) ,$ we write%
\begin{eqnarray*}
&&\frac{2\mu +1}{2}\left\{ \frac{\left( \mathcal{W}_{\lambda ,\mu }\right)
_{n}(z)}{\mathcal{W}_{\lambda ,\mu }(z)}-\frac{2\mu -1}{2\mu +1}\right\} \\
&=&\frac{1+\dsum\limits_{m=1}^{n}a_{m}z^{m}-\left( \frac{2\mu -1}{2}\right)
\dsum\limits_{m=n+1}^{\infty }a_{m}z^{m}}{1+\dsum\limits_{m=1}^{\infty
}a_{m}z^{m}} \\
&=&\frac{1+w(z)}{1-w(z)}.
\end{eqnarray*}%
Therefore%
\begin{equation*}
\left\vert w(z)\right\vert \leq \frac{\left( \frac{2\mu +1}{2}\right)
\dsum\limits_{m=n+1}^{\infty }\left\vert a_{m}\right\vert }{%
2-2\dsum\limits_{m=1}^{n}\left\vert a_{m}\right\vert -\left( \frac{2\mu -3}{2%
}\right) \dsum\limits_{m=n+1}^{\infty }\left\vert a_{m}\right\vert }\leq 1.
\end{equation*}%
The last inequality is equivalent to%
\begin{equation}
\dsum\limits_{m=1}^{n}\left\vert a_{m}\right\vert +\left( \frac{2\mu -1}{2}%
\right) \dsum\limits_{m=n+1}^{\infty }\left\vert a_{m}\right\vert \leq 1.
\label{m7}
\end{equation}%
Since the left hand side of $\left( \ref{m7}\right) $ is bounded above by $%
\left( \frac{2\mu -1}{2}\right) \dsum\limits_{m=1}^{\infty }\left\vert
a_{m}\right\vert ,$ this completes the proof.
\end{proof}

\begin{theorem}
\label{4}Let $\lambda ,$ $\mu \in 
%TCIMACRO{\U{211d} }%
%BeginExpansion
\mathbb{R}
%EndExpansion
,\ $\ with $\lambda >-1$ and $\mu >3.$ Then%
\begin{eqnarray}
\func{Re}\left\{ \frac{\mathcal{W}_{\lambda ,\mu }^{\prime }(z)}{\left( 
\mathcal{W}_{\lambda ,\mu }\right) _{n}^{\prime }(z)}\right\} &\geq &\frac{%
\mu -3}{\mu -1},\text{ \ \ \ }z\in \mathcal{U}.  \label{m8} \\
\func{Re}\left\{ \frac{\left( \mathcal{W}_{\lambda ,\mu }\right)
_{n}^{\prime }(z)}{\mathcal{W}_{\lambda ,\mu }^{\prime }(z)}\right\} &\geq &%
\frac{\mu -1}{\mu +1},\text{ \ \ }z\in \mathcal{U}.  \label{m9}
\end{eqnarray}
\end{theorem}

\begin{proof}
From part (ii) of Lemma \ref{1}, we observe that%
\begin{equation*}
1+\dsum\limits_{m=1}^{\infty }(m+1)\left\vert a_{m}\right\vert \leq \frac{%
\mu +1}{\mu -1},
\end{equation*}%
where $a_{m}=\frac{\Gamma \left( \mu \right) }{m!\Gamma \left( \lambda m+\mu
\right) }.$ This implies that%
\begin{equation*}
\left( \frac{\mu -1}{2}\right) \dsum\limits_{m=1}^{\infty }(m+1)\left\vert
a_{m}\right\vert \leq 1.
\end{equation*}%
Consider%
\begin{eqnarray*}
&&\left( \frac{\mu -1}{2}\right) \left\{ \frac{\mathcal{W}_{\lambda ,\mu
}^{\prime }(z)}{\left( \mathcal{W}_{\lambda ,\mu }\right) _{n}^{\prime }(z)}-%
\frac{\mu -3}{\mu -1}\right\} \\
&=&\frac{1+\dsum\limits_{m=1}^{n}(m+1)a_{m}z^{m}+\left( \frac{\mu -1}{2}%
\right) \dsum\limits_{m=n+1}^{\infty }(m+1)a_{m}z^{m}}{1+\dsum%
\limits_{m=1}^{n}(m+1)a_{m}z^{m}} \\
&=&\frac{1+w(z)}{1-w(z)}.
\end{eqnarray*}%
Therefore%
\begin{equation*}
\left\vert w(z)\right\vert \leq \frac{\left( \frac{\mu -1}{2}\right)
\dsum\limits_{m=n+1}^{\infty }(m+1)\left\vert a_{m}\right\vert }{%
2-2\dsum\limits_{m=1}^{n}(m+1)\left\vert a_{m}\right\vert -\left( \frac{\mu
-1}{2}\right) \dsum\limits_{m=n+1}^{\infty }(m+1)\left\vert a_{m}\right\vert 
}\leq 1.
\end{equation*}%
The last inequality is equivalent to%
\begin{equation}
\dsum\limits_{m=1}^{n}(m+1)\left\vert a_{m}\right\vert +\left( \frac{\mu -1}{%
2}\right) \dsum\limits_{m=n+1}^{\infty }(m+1)\left\vert a_{m}\right\vert
\leq 1.  \label{m10}
\end{equation}%
It suffices to show that the left hand side of $\left( \ref{m10}\right) $ is
bounded above by

$\left( \frac{\mu -1}{2}\right) \dsum\limits_{m=1}^{\infty }\left\vert
a_{m}\right\vert (m+1).$ Which is equivalent to $\frac{\mu -3}{2}%
\dsum\limits_{m=1}^{n}(m+1)\left\vert a_{m}\right\vert \geq 0.$

To prove the result $\left( \ref{m9}\right) ,$ we write%
\begin{eqnarray*}
&&\left( \frac{\mu +1}{2}\right) \left\{ \frac{\left( \mathcal{W}_{\lambda
,\mu }\right) _{n}^{\prime }(z)}{\mathcal{W}_{\lambda ,\mu }^{\prime }(z)}-%
\frac{\mu -1}{\mu +1}\right\} \\
&=&\frac{1+w(z)}{1-w(z)}.
\end{eqnarray*}%
Therefore%
\begin{equation*}
\left\vert w(z)\right\vert \leq \frac{\left( \frac{\mu +1}{2}\right)
\dsum\limits_{m=n+1}^{\infty }(m+1)\left\vert a_{m}\right\vert }{%
2-2\dsum\limits_{m=1}^{n}(m+1)\left\vert a_{m}\right\vert -\frac{\mu -3}{2}%
\dsum\limits_{m=n+1}^{\infty }(m+1)\left\vert a_{m}\right\vert }\leq 1.
\end{equation*}%
The last inequality is equivalent to%
\begin{equation}
\dsum\limits_{m=1}^{n}\left\vert a_{m}\right\vert (m+1)+\frac{\mu -1}{2}%
\dsum\limits_{m=n+1}^{\infty }(m+1)\left\vert a_{m}\right\vert \leq 1.
\label{m11}
\end{equation}%
It suffices to show that the left hand side of $\left( \ref{m11}\right) $ is
bounded above by

$\frac{\mu -1}{2}\dsum\limits_{m=1}^{\infty }(m+1)\left\vert
a_{m}\right\vert ,\ $the proof is complete.
\end{proof}

\begin{theorem}
Let $\lambda ,$ $\mu \in 
%TCIMACRO{\U{211d} }%
%BeginExpansion
\mathbb{R}
%EndExpansion
,\ $\ with $\lambda >-1$ and $\mu >1.$ Then%
\begin{equation}
\func{Re}\left\{ \frac{\mathbb{A}\left[ \mathcal{W}_{\lambda ,\mu }\right]
(z)}{\left( \mathbb{A}\left[ \mathcal{W}_{\lambda ,\mu }\right] \right)
_{n}(z)}\right\} \geq \frac{2\mu -2}{2\mu -1},\text{ \ \ \ }z\in \mathcal{U},
\label{m12}
\end{equation}%
and%
\begin{equation}
\func{Re}\left\{ \frac{\left( \mathbb{A}\left[ \mathcal{W}_{\lambda ,\mu }%
\right] \right) _{n}(z)}{\mathbb{A}\left[ \mathcal{W}_{\lambda ,\mu }\right]
(z)}\right\} \geq \frac{2\mu -1}{2\mu },\text{ \ \ \ }z\in \mathcal{U},
\label{m13}
\end{equation}%
where $\mathbb{A}\left[ \mathcal{W}_{\lambda ,\mu }\right] $ is the
Alexander transform of $\mathcal{W}_{\lambda ,\mu }.$
\end{theorem}

\begin{proof}
To prove $\left( \ref{m12}\right) ,\ $we consider from part (iii) of Lemma %
\ref{1} so that%
\begin{equation*}
1+\dsum\limits_{m=1}^{\infty }\frac{\left\vert a_{m}\right\vert }{\left(
m+1\right) }\leq \frac{2\mu }{2\mu -1},
\end{equation*}%
which is equvalent to%
\begin{equation*}
\left( 2\mu -1\right) \dsum\limits_{m=1}^{\infty }\frac{\left\vert
a_{m}\right\vert }{\left( m+1\right) }\leq 1,
\end{equation*}%
where $a_{m}=\frac{\Gamma \left( \mu \right) }{m!\Gamma \left( \lambda m+\mu
\right) }.\ $Now, we write%
\begin{eqnarray*}
&&\left( 2\mu -1\right) \left\{ \frac{\mathbb{A}\left[ \mathcal{W}_{\lambda
,\mu }\right] (z)}{\left( \mathbb{A}\left[ \mathcal{W}_{\lambda ,\mu }\right]
\right) _{n}(z)}-\frac{2\mu -2}{2\mu -1}\right\} \\
&=&\frac{1+\dsum\limits_{m=1}^{n}\frac{a_{m}}{\left( m+1\right) }%
z^{m}+\left( 2\mu -1\right) \dsum\limits_{m=n+1}^{\infty }\frac{a_{m}}{%
\left( m+1\right) }z^{m}}{1+\dsum\limits_{m=1}^{n}\frac{a_{m}}{\left(
m+1\right) }z^{m}} \\
&=&\frac{1+w(z)}{1-w(z)},
\end{eqnarray*}%
where%
\begin{equation*}
\left\vert w(z)\right\vert \leq \frac{\left( 2\mu -1\right)
\dsum\limits_{m=n+1}^{\infty }\frac{\left\vert a_{m}\right\vert }{\left(
m+1\right) }}{2-2\dsum\limits_{m=1}^{n}\frac{\left\vert a_{m}\right\vert }{%
\left( m+1\right) }-\left( 2\mu -1\right) \dsum\limits_{m=n+1}^{\infty }%
\frac{\left\vert a_{m}\right\vert }{\left( m+1\right) }}\leq 1.
\end{equation*}%
The last inequality is equivalent to%
\begin{equation}
\dsum\limits_{m=1}^{n}\frac{\left\vert a_{m}\right\vert }{\left( m+1\right) }%
+\left( 2\mu -1\right) \dsum\limits_{m=n+1}^{\infty }\frac{\left\vert
a_{m}\right\vert }{\left( m+1\right) }\leq 1.  \label{m14}
\end{equation}%
It suffices to show that the left hand side of $\left( \ref{m14}\right) $ is
bounded above by

$(2\mu -1)\dsum\limits_{m=1}^{\infty }\frac{\left\vert a_{m}\right\vert }{%
\left( m+1\right) },$ which is equivalent to $(2\mu
-2)\dsum\limits_{m=1}^{\infty }\frac{\left\vert a_{m}\right\vert }{\left(
m+1\right) }\geq 0.$ This completes the proof.

The proof of $\left( \ref{m13}\right) $ is similar to the proof of Theorem %
\ref{3}.
\end{proof}

\textbf{Remark 2.4.} \label{rm1} For $\lambda =1,$ $\mu =5/2$ we get $%
\mathcal{W}_{1,5/2}(-z)=\frac{3}{4}\left( \frac{\sin (2\sqrt{z})}{2\sqrt{z}}%
-\cos (2\sqrt{z})\right) ,$ and for $n=0,$ we have $\left( \mathcal{W}%
_{1,5/2}\right) _{0}(z)=z,$ so, 
\begin{equation}
\func{Re}\left( \frac{\sin (2\sqrt{z})-2\sqrt{z}\cos (2\sqrt{z})}{2z\sqrt{z}}%
\right) \geq \frac{2}{3}\ \ \ \left( z\in \mathcal{U}\right) ,  \label{r1}
\end{equation}%
and%
\begin{equation}
\func{Re}\left( \frac{2z\sqrt{z}}{\sin (2\sqrt{z})-2\sqrt{z}\cos (2\sqrt{z})}%
\right) \geq \frac{1}{2}\ \ \ \left( z\in \mathcal{U}\right) .  \label{r2}
\end{equation}%
The image domains of $f(z)=\frac{\sin (2\sqrt{z})-2\sqrt{z}\cos (2\sqrt{z})}{%
2z\sqrt{z}}$ and $g(z)=\frac{2z\sqrt{z}}{\sin (2\sqrt{z})-2\sqrt{z}\cos (2%
\sqrt{z})}$ are shown in Figure 1.%
\begin{equation*}
\FRAME{itbpF}{3.4134in}{1.7893in}{0in}{}{}{Figure}{\special{language
"Scientific Word";type "GRAPHIC";maintain-aspect-ratio TRUE;display
"USEDEF";valid_file "T";width 3.4134in;height 1.7893in;depth
0in;original-width 5.073in;original-height 2.6455in;cropleft "0";croptop
"1";cropright "1";cropbottom "0";tempfilename
'NQCMVY01.wmf';tempfile-properties "XPR";}}
\end{equation*}

\section{Partial Sums of $\mathbb{W}_{\protect\lambda ,\protect\mu }(z)$}

\begin{theorem}
Let $\lambda ,$ $\mu \in 
%TCIMACRO{\U{211d} }%
%BeginExpansion
\mathbb{R}
%EndExpansion
,\ $\ with $\lambda >-1$ and $\mu +\lambda >1.$ Then%
\begin{equation}
\func{Re}\left\{ \frac{\mathbb{W}_{\lambda ,\mu }(z)}{\left( \mathbb{W}%
_{\lambda ,\mu }\right) _{n}(z)}\right\} \geq \frac{2\left( \lambda +\mu
\right) -2}{2\left( \lambda +\mu \right) -1}\text{, \ \ }z\in \mathcal{U},
\label{m15}
\end{equation}%
and%
\begin{equation}
\func{Re}\left\{ \frac{\left( \mathbb{W}_{\lambda ,\mu }\right) _{n}(z)}{%
\mathbb{W}_{\lambda ,\mu }(z)}\right\} \geq \frac{2\left( \lambda +\mu
\right) -1}{2\left( \lambda +\mu \right) },\text{ \ \ }z\in \mathcal{U},
\label{m16}
\end{equation}%
where $\mathbb{W}_{\lambda ,\mu }(z)$ is the normalized Wright function.
\end{theorem}

\begin{proof}
By using Lemma $\ref{2}$ (i), It is clear that 
\begin{equation*}
1+\dsum\limits_{m=1}^{\infty }\left\vert a_{m}\right\vert \leq \frac{2\left(
\lambda +\mu \right) }{2\left( \lambda +\mu \right) -1},
\end{equation*}%
where $a_{m}=\frac{\Gamma \left( \lambda +\mu \right) }{\left( m+1\right)
!\Gamma \left( \lambda m+\lambda +\mu \right) }.$ This implies that%
\begin{equation*}
\left\{ 2\left( \lambda +\mu \right) -1\right\} \dsum\limits_{m=1}^{\infty
}\left\vert a_{m}\right\vert \leq 1.
\end{equation*}%
Now we may write%
\begin{eqnarray*}
&&\left\{ 2\left( \lambda +\mu \right) -1\right\} \left\{ \frac{\mathbb{W}%
_{\lambda ,\mu }(z)}{\left( \mathbb{W}_{\lambda ,\mu }\right) _{n}(z)}-\frac{%
2\left( \lambda +\mu \right) -2}{2\left( \lambda +\mu \right) -1}\right\} \\
&=&\frac{1+\dsum\limits_{m=1}^{n}a_{m}z^{m}+\left\{ 2\left( \lambda +\mu
\right) -1\right\} \dsum\limits_{m=n+1}^{\infty }a_{m}z^{m}}{%
1+\dsum\limits_{m=1}^{n}a_{m}z^{m}} \\
&=&\frac{1+w(z)}{1-w(z)}.
\end{eqnarray*}%
It is clear that%
\begin{equation*}
w(z)=\frac{\left\{ 2\left( \lambda +\mu \right) -1\right\}
\dsum\limits_{m=n+1}^{\infty }a_{m}z^{m}}{2+2\dsum%
\limits_{m=1}^{n}a_{m}z^{m}+\left\{ 2\left( \lambda +\mu \right) -1\right\}
\dsum\limits_{m=n+1}^{\infty }a_{m}z^{m}},
\end{equation*}%
and%
\begin{equation*}
\left\vert w(z)\right\vert \leq \frac{\left\{ 2\left( \lambda +\mu \right)
-1\right\} \dsum\limits_{m=n+1}^{\infty }\left\vert a_{m}\right\vert }{%
2-2\dsum\limits_{m=1}^{n}\left\vert a_{m}\right\vert -\left\{ 2\left(
\lambda +\mu \right) -1\right\} \dsum\limits_{m=n+1}^{\infty }\left\vert
a_{m}\right\vert }.
\end{equation*}%
This implies that $\left\vert w\left( z\right) \right\vert \leq 1$ if and
only if%
\begin{equation}
\dsum\limits_{m=1}^{n}\left\vert a_{m}\right\vert +\left\{ 2\left( \lambda
+\mu \right) -1\right\} \dsum\limits_{m=n+1}^{\infty }\left\vert
a_{m}\right\vert \leq 1.  \label{m17}
\end{equation}%
It suffices to show that the left hand side of $\left( \ref{m17}\right) $ is
bounded above by

$\left\{ 2\left( \lambda +\mu \right) -1\right\} \dsum\limits_{m=1}^{\infty
}\left\vert a_{m}\right\vert ,\ $which is equivalent to $\left\{ 2\left(
\lambda +\mu \right) -2\right\} \dsum\limits_{m=1}^{\infty }\left\vert
a_{m}\right\vert \geq 0.$

To prove $\left( \ref{m16}\right) ,$ we consider that%
\begin{equation*}
2\left( \lambda +\mu \right) \left\{ \frac{\left( \mathbb{W}_{\lambda ,\mu
}\right) _{n}(z)}{\mathbb{W}_{\lambda ,\mu }(z)}-\frac{2\left( \lambda +\mu
\right) -1}{2\left( \lambda +\mu \right) }\right\} .
\end{equation*}%
\begin{eqnarray*}
&=&\frac{1+\dsum\limits_{m=1}^{n}a_{m}z^{m}+\left\{ 2\left( \lambda +\mu
\right) -1\right\} \dsum\limits_{m=n+1}^{\infty }a_{m}z^{m}}{%
1+\dsum\limits_{m=1}^{\infty }a_{m}z^{m}} \\
&=&\frac{1+w(z)}{1-w(z)}.
\end{eqnarray*}%
Therefore%
\begin{equation*}
\left\vert w(z)\right\vert \leq \frac{\left\{ 2\left( \lambda +\mu \right)
\right\} \dsum\limits_{m=n+1}^{\infty }\left\vert a_{m}\right\vert }{%
2-2\dsum\limits_{m=1}^{n}\left\vert a_{m}\right\vert -\left\{ 2\left(
\lambda +\mu \right) -2\right\} \dsum\limits_{m=n+1}^{\infty }\left\vert
a_{m}\right\vert }.
\end{equation*}%
The last inequality is equivalent to%
\begin{equation}
\dsum\limits_{m=1}^{n}\left\vert a_{m}\right\vert +\left\{ 2\left( \lambda
+\mu \right) -1\right\} \dsum\limits_{m=n+1}^{\infty }\left\vert
a_{m}\right\vert \leq 1.  \label{m18}
\end{equation}%
Since the left hand side of $\left( \ref{m18}\right) $ is bounded above by $%
\left\{ 2\left( \lambda +\mu \right) -1\right\} \dsum\limits_{m=1}^{\infty
}\left\vert a_{m}\right\vert ,\ $the proof is complete.
\end{proof}

Similarly, we have the following result.

\begin{theorem}
\label{5}Let $\lambda ,$ $\mu \in 
%TCIMACRO{\U{211d} }%
%BeginExpansion
\mathbb{R}
%EndExpansion
,\ $\ with $\lambda >-1$ and $\mu +\lambda >\frac{3}{2}.$ Then%
\begin{equation}
\func{Re}\left\{ \frac{\mathbb{W}_{\lambda ,\mu }^{\prime }(z)}{\left( 
\mathbb{W}_{\lambda ,\mu }\right) _{n}^{\prime }(z)}\right\} \geq \frac{%
2\left( \lambda +\mu \right) -3}{2\left( \lambda +\mu \right) -1}\text{, \ \ 
}z\in \mathcal{U},  \label{m19}
\end{equation}%
and%
\begin{equation}
\func{Re}\left\{ \frac{\left( \mathbb{W}_{\lambda ,\mu }\right) _{n}^{\prime
}(z)}{\mathbb{W}_{\lambda ,\mu }^{\prime }(z)}\right\} \geq \frac{2\left(
\lambda +\mu \right) -1}{2\left( \lambda +\mu \right) +1},\text{ \ \ }z\in 
\mathcal{U},  \label{m20}
\end{equation}%
where $\mathbb{W}_{\lambda ,\mu }(z)$ is the normalized Wright function.
\end{theorem}

\begin{proof}
Proof is similar to the Theorem \ref{4}.
\end{proof}

Recently Ravichandran \cite{Ravi} presented a survey article on geometric
properties of partial sums of univalent functions. Using Noshiro-Warschawski
Theorem \cite{good} for $n=0$ in the inequalities $\left( \ref{m8}\right) $
of Theorem \ref{4} and $\left( \ref{m19}\right) $ of Theorem \ref{5}, the
functions $\mathcal{W}_{\lambda ,\mu }(z)$ and $\mathbb{W}_{\lambda ,\mu
}(z) $ are univalent and also close to convex. Noshiro \cite{nos} showed
that the radius of starlikness of $f_{n}$ ( the partial sums of the function 
$f\in \mathcal{A}$) is $1/M$ if $f$ satisfies the inequality $\left\vert
f^{\prime }(z)\right\vert \leq M.$ This implies that by using the parts (ii)
of Lemma \ref{1} and Lemma \ref{3}, the radii of starlikeness of the
functions $\left( \mathcal{W}_{\lambda ,\mu }\right) _{n}(z)$ and $\left( 
\mathbb{W}_{\lambda ,\mu }\right) _{n}(z)$ are $\frac{\mu -1}{\mu +1}$ and $%
\frac{2\left( \lambda +\mu \right) -1}{2\left( \lambda +\mu \right) +1}$
respectively.

\textbf{Acknowledgement:} \textit{The research of N. Ya\u{g}mur is supported
by Erzincan University Rectorship under "The Scientific and Research Project
of Erzincan University", Project No: FEN-A-240215-0126.}

\end{document}